\documentclass[a4paper,12pt]{article}

\usepackage{amsmath,amsthm,amssymb}
\usepackage{enumerate}
\usepackage{graphicx}
\usepackage{titlesec}
\usepackage{empheq}
\usepackage{color}
\usepackage{xfrac,bigints}
\usepackage{comment}
\usepackage[normalem]{ulem}

\newtheorem{Th}{Theorem}[section]
\newtheorem{lemma}[Th]{Lemma}
\newtheorem{remark}[Th]{Remark}
\newtheorem{Prop}[Th]{Proposition}

\newtheorem{Def}[Th]{Definition}

\newtheorem{Rem}[Th]{Remark}

\newtheorem{a priori condition}[Th]{A priori condition}

\newtheorem{assumption}[Th]{Assumption}

\titleformat*{\section}{\normalsize\bfseries}
\titleformat*{\subsection}{\normalsize\bfseries}
\renewcommand{\thesection}{{{\normalsize\arabic{section}}}}

\renewcommand{\theequation}{{{\thesection.\arabic{equation}}}}

\newcommand{\eproof}{\mbox{\ }~\hfill
\mbox{\large $\Box$}}

\newcommand{\pf}{\noindent{\bf Proof}}

\def\N{{\mathbb N}}
\def\R{{\mathbb R}}

\def\DS{\displaystyle}

\def\dvg{\mbox{\rm div}}

\definecolor{purple}{rgb}{.75, 0, .25}

\makeatletter
\long\def\@makefntext#1{\parindent 1em\noindent
\@hangfrom{\hbox to 1.8em{\hss$^{\@thefnmark}$}}#1}
\makeatother

\begin{document}

\title{Anisotropic extended Burgers model, its relaxation tensor and properties of the associated Boltzmann viscoelastic system}

\author{Maarten V. de Hoop \thanks{Simons Chair in Computational and Applied
Mathematics and Earth Science, Rice University, Houston TX, USA (Email: mdehoop@rice.edu).}\and
\qquad Masato Kimura\thanks{Faculty of Mathematics and Physics, Kanazawa University, Kanazawa 920-1192, Japan (mkimura@se.kanazawa-u.ac.jp).}\and
Ching-Lung Lin\thanks{Department of Mathematics, National Cheng-Kung University, Tainan 701, Taiwan (Email:
cllin2@mail.ncku.edu.tw).}\and
Gen Nakamura\thanks{Department of
Mathematics, Hokkaido University, Sapporo 060-0808, Japan and Research Center of Mathematics for Social Creativity, Research Institute for Electronic Science, Hokkaido University, Sapporo 060-0812, Japan (Email: gnaka@math.sci.hokudai.ac.jp).
}\and
Kazumi Tanuma\thanks{Department of Mathematics, Faculty of Science and Technology, Gunma University, Kiryu 376-8515, Japan (Email: tanuma@gunma-u.ac.jp).}}

\maketitle

\begin{abstract}
We provide a new method for constructing the anisotropic relaxation tensor and proving its exponential decay property for the extended Burgers model (abbreviated by EBM). The EBM is an important viscoelasticity model in rheology, and used in Earth and planetary sciences. Upon having this tensor, the EBM can be converted to a Boltzmann-type viscoelastic system of equations (abbreviated by BVS). Historically, the relaxation tensor for the EBM is derived by solving the constitutive equation using the Laplace transform. (We refer to this approach by the L-method.) Since inverting the inverse Laplace transform needs a partial fractions expansion, the L-method needs to assume that the EBM elasticity tensors satisfy a commutivity condition. The new method not only avoids this condition but also enables obtaining several important properties of the relaxation tensor, including its positivity, smoothness with respect to the time variable, its exponential decay property together with its derivative, and its causality. Furthermore, we show that the BVS converted from the EBM has the exponential decay property. That is, any solution for its initial boundary value problem with homogeneous boundary data and source decays exponentially as time tends to infinity.

\bigskip


\end{abstract}

\renewcommand{\theequation}{\thesection.\arabic{equation}}

\section{Introduction}
\label{sec1}
\setcounter{equation}{0}


The anisotropic extended Burgers model (EBM) is a model of viscoelasticity inherently constructed from springs and dashpots.
The EBM is a parametric model widely applied in the Earth and planetary sciences \cite{Hetland, Ivins}. It is fundamentally distinct from the anisotropic extended Maxwell model, which can be easily converted into a Boltzman-type viscoelastic system of equations, abbreviated by BVS \cite{HKLN}. 

The constitutive equation between the stress $\sigma(x,t)$ and strain $e(x,t)$ of a Boltzman-type system is given in a form where $\sigma(x,t)$ is expressed in terms of an integral of the time derivative $\dot e(x,t)$ with a kernel, $G$, identified as the relaxation tensor. That is,
\begin{equation*}\label{Boltzmann constitutive eq}
\sigma(x,t)=\int_0^t G(x,t,s)\,\dot e(x,s)\,ds .
\end{equation*}
We note that this is a rough way to introduce the relaxation tensor. (We refer to the second paragraph after Assumption~\ref{fundamental assumption} for the precise explanation.)

In this paper, we will obtain the relaxation tensor for the EBM in the convolutional form:
\begin{equation*}\label{convolutional form}
G(x,t,s)=G(x,t-s),\,\,0\le s\le t .
\end{equation*}
In general, the relaxation tensor does not have to be of the convolutional form. The relaxation tensor of the general type, $G(x,t,s)$, was considered before in \cite{Dafermos}. 


As suggested by Anderson and Minster \cite{AndersonMinster_1979}, Andrade’s rheology and fractional power law can be interpreted as a continuum of Kelvin-Voigt elements in series, which can itself be approximated by a discrete series of Kelvin-Voigt elements (e.g., Birk and Song \cite{BirkSong_2010} and Ben Jazia \textit{et al}. \cite{JaziaLombardBellis2013}). The resulting model is known as the extended Burgers model.

While capturing the full range of the observed mechanical behavior of the upper mantle materials and reconciling the behavior of the mantle across the ``complete'' spectrum of geodynamic time-scales, the EBM has emerged as the model of choice \cite{YuenPeltier_1982}. Dissipation data with frequency dependence of attenuation across, say, six orders of magnitude, are indeed typically best fitted to this model. The EBM yields a consistent description across normal modes, seismic body waves and body tides, and from laboratory ultrasound to the Wilson cycle; see Lau and Faul \cite{LauFaul_2019}. It also captures the long time scale of mantle convection on which Earth behaves viscously. 
The Andrade creep model, and EBM, have been introduced in the astronomical community by Efroimsky and Lainey \cite{EfroimskyLainey_2007}; it has been applied, for example, to Enceladus' dissipation of energy due to forced libration (see Gevorgyan \textit{at al}. \cite{GevorgyanBoueRagazzoRuizCorreia_2020}).


Notably, though anisotropy is widely identified in purely elastic model descriptions of the mantle, in the geophysics literature it has not yet been included in the EBM description in view of the challenges that it has posed.  
The authors are not aware of any systematic methods to derive the relaxation tensor---that is, a Boltzmann-type representation---for the EBM, especially in the anisotropic case.
Carcione \cite{Carcione} studied anisotropic viscoelasticity in the context of wave propagation and introduced a condition that reappears in our analysis using spectral representations, described below in Section \ref{section 3}.

The contributions of this paper are the following:
\begin{enumerate}[(i)]
\item
A new method to construct the relaxation tensor under weak conditions, while developing its properties.
\item
A proof of exponential decay of solutions $t \rightarrow \infty$ of the initial boundary value problem of the BVS with homogeneous boundary data and source term, obtained with the relaxation tensor for the EBM.
\end{enumerate}

\noindent
The remainder of this paper is organized as follows. In Section 2, we introduce the EBM, its relaxation tensor, and the constitutive equation in an integro-differential form. In Section 3, we make a brief visit to the L-method which needs a commutativity assumption for the relevant elasticity tensors. In Section 4, an ordinary differential system of equations related to the EBM is introduced, and we observe that this is equivalent to the system of integro-differential equations introduced in the previous section. We also introduce tensor block matrices playing a role in Section 5, which is devoted to proving the mentioned aims, (i) and (ii) above. The last section contains the conclusion and a discussion.

\section{EBM and BVS}\label{EBM etc}
\setcounter{equation}{0}

In this section, we introduce the EBM and relaxation tensor. Then, we present the integro-differential equations for the total strain and total stress. The usual L-method given in Section \ref{section 3} uses these equations and the Laplace transform to derive the relaxation tensor assuming that the elasticity tensors for the springs satisfy a commutativity condition.
At the end of this section, we give an initial boundary value problem for the BVS assuming that we have constructed the relaxation tensor, and give a sufficient condition such that solutions of the mentioned problem decay exponentially in time. 

We let $\Omega$ be a bounded domain in the $d$-dimensional Euclidean space ${\mathbb R}^d$ with $d=2,3$. We assume that the boundary $\Gamma=\partial\Omega$ of $\Omega$ is Lipschitz smooth and denote its outward unit normal by $\nu\in \R^d$. 
We interpret $\Omega$ as a reference domain on which we consider a small viscoelastic deformation. 

Now, for a while, we will ignore the smoothness of the functions and just focus on introducing more notation used throughout the paper. We denote a displacement for a small deformation of $\Omega$ by $u=u(x,t)\in\R^d$ for a position a.e. $x\in \Omega$ and a time $t\in [0,T)$, where $T\in (0,\infty ]$. We also define the infinitesimal strain tensor by ${e[u]:=\frac{1}{2} \{\nabla u +(\nabla u^T)^{\mathfrak{t}}\} : \Omega\times [0,T)\to \R^{d\times d}_\text{sym}}$, where $\nabla u$ is the gradient of $u$ with respect to $x$ and $(\nabla u^T)^{\mathfrak{t}}$ is the transpose of $\nabla u$.

Further, we give the following definitions for matrices and tensors which will be used throughout the paper. For matrices $A=(a_{pq}),~B=(b_{pq})\in \R^{d\times d}$, their Frobenius inner product and the norm $|A|$ are defined by
$A:B:=a_{pq}b_{pq}$, $|A|:=\sqrt{A:A}$, respectively. Here, we have used Einstein's summation convention. For a matrix $\xi=(\xi_{pq})\in \R^{d\times d}$ and
a fourth-order tensor 
$C=(c_{pqrs})\in \R^{d\times d\times d\times d}$,
their product is denoted by
$C\xi:=(c_{pqrs}\xi_{rs})\in\R^{d\times d}$.
The product of two fourth-order tensors $C=(c_{pqrs})$ and $D=(d_{pqrs})$ is defined by $CD=(c_{pqkl}d_{klrs})$ as well.
Also, the set of $d$-dimensional symmetric matrices is represented by $\R^{d\times d}_\text{sym}$.
For a fourth-order tensor 
$C=(c_{pqrs})$, it is said that the major symmetry condition is satisfied when $c_{pqrs}=c_{rspq}$ holds and that the minor symmetry condition is satisfied when $c_{pqrs}=c_{qprs}=c_{pqsr}$
holds, for all $p,q,r,s=1,2,\cdots,d$. Define the transpose $C^{\mathfrak{t}}$ of $C=(c_{pqrs})$ by $C^{\mathfrak{t}}:=(C_{rspq})$. Then, the major symmetry can be given as $C^{\mathfrak{t}}=C$. 
When the major and minor symmetries are both satisfied for $C$, it is said that full symmetry is satisfied for $C$, and we denote $C\in{\mathbb R}_{\text{f-sym}}^{d\times d\times d\times d}$. For fourth-order tensors $C,\,D$ satisfying the major symmetry, we write $C\le D$ if
$(C\xi):\xi\le(D\xi):\xi$ for all $\xi\in{\mathbb R}^{d\times d}$ is satisfied. We remark that a linear map from $\R^{d\times d}_\text{sym}$ into itself is uniquely represented by
a fourth-order tensor $C$ with the minor symmetry
as $\R^{d\times d}_\text{sym}\ni \xi\longmapsto C\xi\in \R^{d\times d}_\text{sym}$.

From this point on, we assume that the displacement $u=u(x,t)\in C^2([0,T); H^1(\Omega;{\mathbb R}^d))$, the total strain $e=e(x,t)\in C^1([0,T); L^2(\Omega;{\mathbb R}^{d\times d}_{\text{sym}}))$, and the total stress \\ $\sigma(x,t)\in C^1([0,T); L^2(\Omega;{\mathbb R}^{d\times d}_{\text{sym}}))$. Here, $H^1(\Omega;{\mathbb R}^d)$ is the set of all ${\mathbb R}^d$-valued functions in the $L^2$-based Sobolev space of order $1$. Also, for any Banach space $S$, such as ${\mathbb R}$ and ${\mathbb R}_{\text{sym}}^{d\times d\times d\times d}$, $L^p(\Omega;S)$ with $1\le p\le\infty$ and $C^m([0,T); S)$ with $m\in{\mathbb N}\cup\{0\}$ are the sets of all $S$-valued $L^p$-functions on $\Omega$ and the sets of all $S$-valued $m$-times continuously differentiable functions defined on $[0,T)$, respectively. Then, the EBM is a spring-dashpot model connecting one Maxwell model $M_0$ and $n$-number of Kelvin-Voigt models $K_i,\,i=1,2,\cdots,n$ in series (see Figure 1). 
\begin{figure}[htb]
\centering
\fbox{
\includegraphics[width=0.54\textwidth]{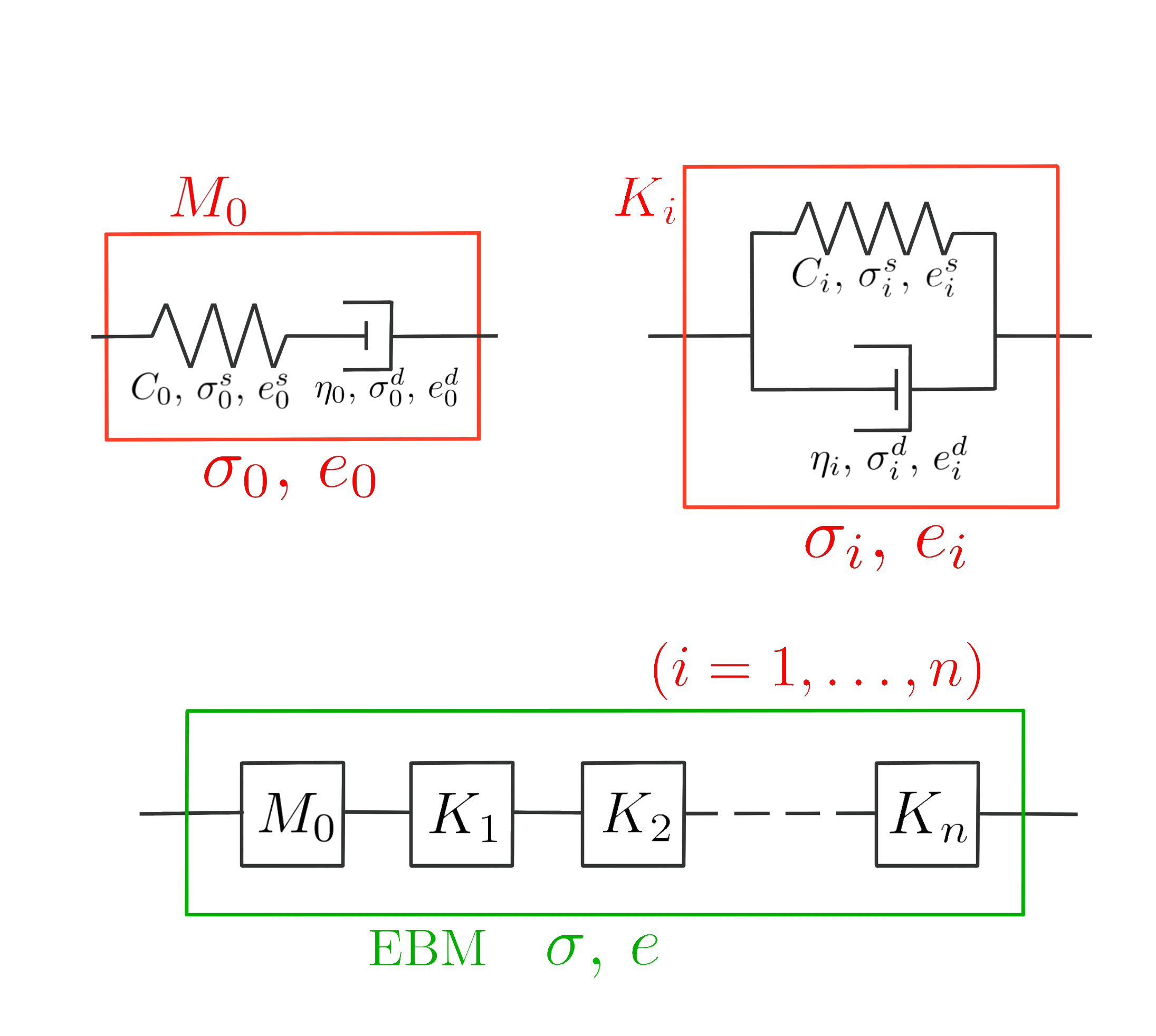}}
\caption{Extended Burgers model: $M_0$: Maxwell component, $K_i$: Kelvin-Voigt component $(i=1,\cdots,n)$.}
\label{fig:EBM}
\end{figure}
Corresponding to each of these models, we denote its pair of strain and stress as $(e_i,\sigma_i)\in C^1([0,T); L^2(\Omega;{\mathbb R}^{d\times d}_{\text{sym}}))\times C^1([0,T); L^2(\Omega;{\mathbb R}^{d\times d}_{\text{sym}}))$ in accordance with the labeling $i=0,1,\cdots,n$. Likewise, we denote their pair of elasticity tensor and viscosity $(C_i,\eta_i)\in L^\infty(\Omega;{\mathbb R}^{d\times d\times d\times d})$ $\times L^\infty(\Omega;{\mathbb R})$ as well.  Then, their relation for each model and the equation of motion for the EBM are given as follows
\begin{align}
M_0~&:~
\begin{cases}
  \sigma_0^s=C_0e_0^s,\quad \eta_0\dot{e}_0^d=\sigma_0^d\\
  e_0=e_0^s+e_0^d,\quad \sigma_0=\sigma_0^s=\sigma_0^d
\end{cases} \label{M0}\\
K_i~&:~
\begin{cases}
  \sigma_i^s=C_ie_i^s,\quad \eta_i\dot{e}_i^d=\sigma_i^d\\
  e_i=e_i^s=e_i^d,\quad \sigma_i=\sigma_i^s+\sigma_i^d
\end{cases}
\quad (i=1,\cdots,n)\label{Ki}\\
\mbox{EBM}~&:~
\begin{cases}
  e=e[u]\\
  \rho \ddot{u}=\dvg\, \sigma \\
  \DS{e=\sum_{i=0}^ne_i,\quad \sigma=\sigma_0=\sigma_1=\cdots =\sigma_n}.
\end{cases}\label{EBM}
\end{align}
Here, ``$\cdot$'' stands for the $t$-derivative.


\medskip
Without duplication of notation, \eqref{M0}, \eqref{Ki}, and \eqref{EBM} can be given more concisely as follows:
\begin{equation}\label{concise EBM}
\left\{
\begin{array}{ll}
e_0=e_0^s+e_0^d,\,\,e=e_0+\displaystyle\sum_{i=1}^n e_i,\\
\\
\sigma=C_0\,e_0^s=\eta_0\,\dot e_0^d=C_i\,e_i+\eta_i\,\dot e_i,\,\, i=1,2,\cdots,n
\end{array}
\right.
\end{equation}

\medskip

\noindent
We impose
\begin{assumption}\label{fundamental assumption}${}$
\begin{itemize}    
\item [{\rm(i)}] There exists $\eta_{\ast}>0$ such that each $\eta_i$ satisfies $\eta_i(x)\ge \eta_{\ast},\, \text{a.e.}\, x\in\Omega$.
\item [{\rm(ii)}] Each $C_i\in L^\infty\left(\Omega;{\mathbb R}_{\text{\rm{f-sym}}}^{d\times d\times d\times d}\right)$. 
\item [{\rm (iii)}] Each $C_i(x)=\left(c^{(i)}_{pqrs}(x)\right)$ satisfies the strong convexity condition:
\[
(C_i(x)\xi):\xi \ge c_*|\xi|^2,\,\,\xi\in\R^{d\times d},\,\, \text{a.e.}\,x\in\Omega
\]
for some $c_*>0$.
\end{itemize}
\end{assumption}

\medskip
We set $\phi_i:=e_i^d$ for $i=0,\cdots,n$
and define $\phi:=(\phi_1,\cdots,\phi_n)$ 
$:~\Omega\times [0,T)\to (\R^{d\times d}_\text{sym})^n$, $\bar\phi:=(\phi_0,\phi)=(\phi_0,\cdots,\phi_n)$ 
$:~\Omega\times [0,T)\to (\R^{d\times d}_\text{sym})^{n+1},$ 
and also set $\psi :=e_0^s : \Omega\times [0,T)\to \R^{d\times d}_\text{sym}$.
It should be remarked here that throughout this paper, spatial variables will not be specified unless necessary. Also, $\text{a.e.}\,x\in\Omega$ will not be mentioned as long as it can be understood from the context.

By ignoring all other variables, the relaxation tensor relates the instantaneous change of total strain, $\dot e(t)$, to the total stress, $\sigma(t)$. This relation can be considered as a convolutional operator $\mathcal{R}$ of a one-port mechanical system (see \cite{Smith} and subsections 10.1-10.3 of \cite{Zemanian}). More precisely, $\mathcal{R}$ is an operator on the space $\mathcal{D}_{R}'$ of right-sided real Schwartz distributions in ${\mathbb R}$; that is, the set of real Schwartz distributions in ${\mathbb R}$ with supports bounded from below by a real constant depending on each distribution. In this setting, we assume $e(t),\,\sigma(t)\in\mathcal{D}_R'$. The mentioned regularities for $e(t),\,\sigma(t)$ are those over the time interval $[0,T)$. Then, the relaxation tensor $G(t)$ is given as a real Schwartz distribution $G(t)$ in ${\mathbb R}$ such that $\mathcal{R}\cdot=G\ast\cdot$ satisfying the causality, $G(t)=0,\,t<0$. Concerning the other variables, we will look for $G(t)$ in the space $\mathcal{D}'({\mathbb R};L^\infty(\Omega; {\mathbb R}^{d\times d\times d\times d}))$, the space of $L^\infty(\Omega;{\mathbb R}^{d\times d\times d\times d})$-valued Schwartz distributions defined in ${\mathbb R}$. Hence, once having the relaxation tensor $G(t)=G(x,t)$, we have the constitutive equation $\sigma=G\ast\dot e$ which is formally written as
\begin{equation}\label{constitutive eq and relaxation}
\left\{\begin{array}{ll}
\sigma(x,t)=\displaystyle{\int_0^t} G(x,t-s)\dot e(x,s)\,ds \\[1em]
\qquad\qquad\text{for any continuously differentiable right-sided $e(x,s)$ in $s$}.
\end{array}
\right.
\end{equation}

Due to the symmetry of the strain tensor $e(x,s)$ and stress tensor $\sigma(t,s)$, it is natural to assume that $G(x,t)$ satisfies the minor symmetry. Furthermore, it turns out that $G(x,t)$ must be smooth with respect to $t\ge0$ for each fixed a.e.\ $x\in\Omega$, and is exponentially decaying as $t\rightarrow\infty$ uniformly with respect to a.e.\ $\,x\in\Omega$. Hence, we could have considered $G(x,t)$ as a $L^\infty(\Omega;{\mathbb R}^{d\times d\times d\times d})$-valued rapidly decreasing function defined on $[0,\infty)$.

\medskip
Summarizing what we have mentioned above, we a priori assume the following conditions for the relaxation tensor $G=G(x,t)\in \mathcal{D}'({\mathbb R};L^\infty(\Omega;{\mathbb R}^{d\times d\times d\times d}))$:
\medskip
\begin{a priori condition}\label{assumption for G}${}$
\begin{itemize}
\item [{\rm (i)}]
The fourth-order tensor $G(x,t)\in\R^{d\times d\times d\times d}$ is zero for $t<0,\,\text{a.e. $x\in\Omega$}$, and it is continuous with respect to $t\ge0$ for each fixed $\text{a.e. $x\in\Omega$}$.
\item [{\rm (ii)}]
$G(x,t)=(G_{ijkl}(x,t))$ satisfies the minor symmetry given as
$$
G_{ijkl}(x,t)=G_{jikl}(x,t)=G_{ijlk}(x,t),\,\,t\ge0,\,\,\text{a.e. $x\in\Omega$}
$$
for $1\le i,j,k,l\le d$.
\end{itemize}
\end{a priori condition}

\bigskip

\noindent
Now, we prepare to gradually move towards getting the relaxation tensor. To begin with, we assume that
\begin{equation}\label{assumption for integro-diff}
e_i(0)=0,\,\,i=1,2,\cdots, n.
\end{equation}
Then, we can derive the following integro-differential equation,
\begin{equation}\label{integro-diff}
\dot e(t)=\left(C_0^{-1}\dot\sigma(t)+\eta_0^{-1}\sigma(t)\right)+\displaystyle
\sum_{i=1}^n\left\{\eta_i^{-1}\sigma(t)-\eta_i^{-2} C_i\int_0^t e^{-(t-s)\eta_i^{-1}C_i}\sigma(s)\,ds\right\}.
\end{equation}
In fact, from \eqref{M0} and \eqref{EBM}, we deduce that
\begin{equation}\label{star1}
\dot e_0=C_0^{-1}\dot\sigma+\eta_0^{-1}\sigma.
\end{equation}
Also, from \eqref{Ki},
\begin{equation}\label{star2}
\dot e_i+\eta_i^{-1}C_ie_i=\eta_i^{-1}\sigma,\,\,\,1\le i\le n.    
\end{equation}
Solving each of these ordinary differential equations for $e_i$ using that $e_i(0)=0$, we find that
\begin{equation}\label{star3}
e_i(t)=\int_0^t e^{-(t-s)\eta_i^{-1}C_i}\eta_i^{-1}\sigma(s)\,ds,\,\,1\le i\le n.
\end{equation}
Then, differentiating these by $t$, we obtain
\begin{equation}\label{star4}
\dot e_i(t)=\eta_i^{-1}\sigma(t)-\eta_i^{-2}C_i\int_0^t e^{-(t-s)\eta_i^{-1}C_i}\sigma(s)\,ds,\,\,1\le i\le n.
\end{equation}
Then, from \eqref{star1}, \eqref{star4} and \eqref{EBM}, we immediately arrive at \eqref{integro-diff}.

\begin{remark}\label{remark for deriving integro-diff}
As a matter of fact, \eqref{integro-diff} follows from just assuming $e(t),\,\sigma(t)\in\mathcal{D}_R'$ ignoring the other variables. This is because the primitive of any $z\in\mathcal{D}_R'$ with $z'\in\mathcal{D}_R'$ is given as $(1_+\ast z)(t)=\int_0^t z(s)ds\,\,(t>0),\,\,=0\,\,(t\le0)$, where $1_+(t)$ is the unit step function defined as 
$$
1_+(t):=\left\{
\begin{array}{ll}
0\,\,&t<0,\\
\frac12\,\,&t=0,\\
1\,\,&t>0
\end{array}
\right.
$$
(see page 274, (2) of \cite{Zemanian}).
    
\end{remark}

\medskip
Next, we will formulate an initial boundary value problem for $(u,\overline\phi)$. For that, let $\Gamma_N$ be a non-empty open subset of $\partial \Omega$, and define $\Gamma_D:=\partial\Omega\setminus\overline{\Gamma_N}$. If $d=3$ and $\Gamma_D\not=\emptyset$, we assume that $\overline{\Gamma_D}\cap\overline{\Gamma_N}$ is Lipschitz smooth. Here, for a set $\Lambda\subset\R^d$,  we denoted its closure by $\overline\Lambda$.
Assume that the displacement $u$ is fixed to zero at $\Gamma_D$, the traction at $\Gamma_N$ is free, and no external force is applied to $\Omega$. 

Then, the initial boundary value problem for the EBM with an initial data $(u^0,\dot u^0,\bar\phi^0)$ for $(u,\dot u,\bar\phi)$ at the initial time $t=0$ is to find $(u,\bar\phi):\Omega\times [0,T)\rightarrow \R^d\times(\R^{d\times d}_\text{sym})^{n+1}$ which satisfies
 
\begin{equation}\label{ebm-prob}
\left\{
\begin{array}{ll}
\rho \ddot{u}=\dvg\, \sigma &\mbox{\rm in}~\Omega\times (0,T),\\
\eta_0\dot{\phi}_0=\sigma &    \mbox{\rm in}~\Omega\times (0,T),\\
\eta_i\dot{\phi}_i=\sigma -C_i\phi_i \quad (i=1,\cdots,n)&    \mbox{\rm in}~\Omega\times (0,T),\\
u=0&\mbox{\rm on}~\Gamma_D\times (0,T),\\
\sigma\nu=0 &\mbox{\rm on}~\Gamma_N\times (0,T),\\
(u, \dot{u}, \bar\phi)|_{t=0}=(u^0, \dot{u}^0, \bar\phi^0)& \mbox{\rm in}~\Omega,
\end{array}
\right.
\end{equation}
where the total stress tensor $\sigma$ is defined by
\begin{align}\label{sigma}
\sigma :=C_0\psi,\quad
  \psi:=
  e[u]-\sum_{i=0}^n\phi_i,
\end{align}
and $\text{``$\mid_{t=0}$''}$ denotes the restriction to $t=0$.

\medskip
The importance of the relaxation tensor is that it converts the EBM to the BVS. In fact, we can give the following initial boundary value problem with mixed-type boundary conditions for the BVS. Suppose that there is no external force, $\Gamma_D$ is fixed, and $\Gamma_N$ is traction free. Then, since $\dot{e}=e[\dot{u}]$ due to $e=e[u]$, an initial and boundary value problem
of the BVS with the constitutive equation \eqref{constitutive eq and relaxation} of the convolutional type is given as follows:
\begin{equation}\label{ebm-prob2}  
\left\{
\begin{array}{ll}
\rho \ddot{u}=\dvg\, \sigma &\mbox{\rm in}~\Omega\times (0,T),\\
{\DS \sigma(x,t)=\int_0^t G(x,t-s)e[\dot{u}](x,s)\,ds} & \mbox{\rm in}~\Omega\times (0,T),\\
u=0&\mbox{\rm on}~\Gamma_D\times (0,T),\\[5pt]
\sigma\nu=0 &\mbox{\rm on}~\Gamma_N\times (0,T),\\[5pt]
(u, \dot{u})|_{t=0}=(u^0, \dot{u}^0)& \mbox{\rm in}~\Omega.
\end{array}
\right.
\end{equation}
Then, according to the study \cite{HLN}, the sufficient conditions for the relaxation tensor $G(x,t)$ to have the solution of \eqref{ebm-prob2} with the initial data $(u^0,\dot u^0)\in H_+ \times H_+$, $H_+:=\{f\in H^1(\Omega):\,f\big|_{\Gamma_D}=0\}$ exponentially decaying as $t\rightarrow\infty$ are given in:

\begin{assumption}{\rm(Exponential decay of solutions)}\label{exponential} 
\begin{itemize}
\item[{\rm (i)}]  $\Gamma_D\not=\emptyset$.
\item[{\rm (ii)}] $K(x):=G(x,0)\in L^\infty(\Omega)$ and $H(x,t):=-\dot G(x,t)\in C^2([0,\infty); L^\infty(\Omega))$.
There exist constants  $\kappa_j>0,\,j=1,\cdots,4$ and $\tilde\kappa_4>0$ such that for $\text{a.e.}\,x\in\Omega,\,\,t\in[0,\infty)$,
\begin{equation}\label{derivatives}
-\kappa_1 H(x,t)\le\dot H(x,t)\le-\kappa_2 H(x,t),\,\ddot H(x,t):=\partial_t^2 H(x,t)\le\kappa_3 H(x,t)
\end{equation}
and
\begin{equation}\label{exp decay of the |G| and |dot G|}
|H(x,t)|+|\dot H(x,t)|\le \kappa_4e^{-\tilde{\kappa}_4 t}.
\end{equation}
\item[{\rm (iii)}] $K(x),\,H(x,t)$ satisfy the major symmetry condition.
\item[{\rm (iv)}]
$K(x),\,H(x,t)$ satisfy
the strong convexity condition.
\item[{\rm (v)}] $\displaystyle{e^{-\kappa_5 T}|\xi|^2\le \left\{\left(K(x)-\int_0^T H(x,t)\,dt\right)\xi \right\}:\xi\le e^{-\kappa_6 T}|\xi|^2,\,\,\xi\in{\mathbb R}_{\text{sym}}^{d\times d},\,\,T>0}$ holds for some constants  $\kappa_5\ge\kappa_6>0$.
\end{itemize}
\end{assumption}

\section{Brief summary of the application of the L-method}\label{section 3}\setcounter{equation}{0}

In this section, we give a brief summary of the application of the L-method. We first clarify the conditions that we invoke in this section for the viscosities of dashpots and elasticity tensors of springs. Besides Assumption~\ref{fundamental assumption}, we impose:

\begin{assumption}\label{homogeneous}${}$
\begin{itemize}
\item [{\rm (i)}] Unless otherwise noted, $\eta_j,\,C_j,\,j=0,1,\cdots,n$ are independent of $\text{a.e. $x\in\Omega$}$.
\item[{\rm (ii)}] {\rm (Commutativity)}
\begin{equation}\label{A.9}
\begin{aligned}
C_jC_l=C_lC_j,  \quad 0\leq j,l \leq n\,\,\,\text{holds}, 
\end{aligned}
\end{equation}
where, for $C_j=\left(c^{(j)}_{pqrs}(x)\right)$, $C_l=\left(c^{(l)}_{pqrs}(x)\right)$, the multiplication $B=(b_{pqrs})=C_j C_l$ is defined by
$$
b_{pqrs}=c^{(j)}_{pq\alpha\beta}c^{(l)}_{\alpha\beta rs},\,\,1\le p,q,r,s\le d.
$$
\end{itemize}

\end{assumption}

\medskip
We consider each $C_l$ as a linear map $\R^{d\times d}_\text{sym}\ni \xi\longmapsto C_l\xi\in \R^{d\times d}_\text{sym}$. Then, from Assumption \ref{homogeneous}, each $C_l$ can have the following spectral representation
\begin{equation}\label{A.10}
\begin{aligned}
C_l=\sum_{k=1}^{\tilde d} \lambda_k^{(l)}P_k\,\,\text{with each $\lambda_k^{(l)}>0$},\,\,\tilde d:=d(d+1)/2,  
\end{aligned}
\end{equation}
where each $P_k$ is the common spectral projection which does not depend on $l=0,1,\cdots,n$. The common projections were used in \cite{Browaeys, Carcione} for analyzing waves in anisotropic elastic media.

We suppose that $e(t),\,\sigma(t)$ are right-sided tempered distribution defined over ${\mathbb R}$ with supports bounded from below at $t=0$. Note that this is consistent to the argument for causality given in Section \ref{EBM etc}.
Then, taking the Laplace transform of \eqref{integro-diff}, we have
\begin{equation}\label{Laplace trf const_eq}
s\, \hat e=M(s)\hat\sigma,\,\,s>0
\end{equation}
with the tensorial factor $M(s)$ given by
\begin{equation}
\label{M(s)-original-1}
M(s):=a_0 I+s C_0^{-1}+\displaystyle\sum_{i=1}^n a_i s(sI+a_i C_i)^{-1}
\end{equation}
with $a_i:=\eta_i^{-1}>0$
for $i=0,1,\cdots,n$, where $\hat e,\,\hat \sigma$ are the Laplace transforms of $e,\,\sigma$.

In \eqref{M(s)-original-1}, if $a_iC_i=a_jC_j$ for a single pair $(i,j)$ with $i<j$, then we have 
\begin{equation*}
a_i s(sI+a_i C_i)^{-1}+ a_j s(sI+a_j C_j)^{-1} \\
= 2a_i s(sI+a_i C_i)^{-1}=\tilde{a_i} s(sI+\tilde{a_i} \tilde{C_i})^{-1} ,
\end{equation*}
with $\tilde{a_i}=2a_i$ and $\tilde{C_i}=\frac{1}{2}C_i$. Thus the formula $M(s)$ in \eqref{M(s)-original-1} will only change $n$ to $n-1$. Since we can argue similarly when several $a_iC_i$'s coincide, to simplify the argument, we will assume that for each $k\ (1\le k\le \tilde d)$,  $A_l^k=a_l\lambda_k^{(l)}\ (l=1,2,\cdots, n)$ are all distinct.

On the other hand, from the second equation of the Laplace transformed $\eqref{ebm-prob2}$, we observe that $\widehat{G}(s):=(\mathcal{L}G)(s)$, the Laplace transform of the relaxation tensor $G(t)$, satisfies 
\begin{equation}\label{A.11}
\begin{aligned}
\hat{\sigma}(s)=\hat{G}(s)\, s\, \hat{e}.  
\end{aligned}
\end{equation}
It then follows that 
\begin{equation}
\label{M(s)-G-hat}
\widehat{G}(s)=M(s)^{-1},\,\,s>0.
\end{equation}
Hence, the remaining tasks are to compute the inverse Laplace transform of $M(s)^{-1}$ and examine its properties.
Substituting \eqref{A.10} into \eqref{M(s)-original-1} and \eqref{M(s)-G-hat}, we see that
\begin{equation}\label{A.12}
\begin{aligned}
\widehat{G}(s)=\sum_{k=1}^{\tilde d} \left(a_0+\left(\lambda_k^{(0)}\right)^{-1}\, s+  \sum_{i=1}^n a_i\, s \left(s+a_i \lambda_k^{(i)} \right)^{-1}\right)^{-1} P_k
\end{aligned}
\end{equation}
with $a_i:=\eta_i^{-1}$.

Now, we define
$$D(s;k)=a_0+ \left(\lambda_k^{(0)} \right)^{-1}s+  \sum_{i=1}^n a_i\, s \left(s+a_i \lambda_k^{(i)} \right)^{-1},\,\,k=1,2,\cdots,\tilde d.$$
To prove the property of $G(t)$ as $t\rightarrow\infty$, the following lemma is essential.

\begin{lemma}\label{lemA.3}
Let $s=p+iq$ be the roots of $D(s;k)=0$, where $p,\,q \in \mathbb{R}$. Then, we have $p<0,\, q=0$.
\end{lemma}
\proof
For $s=p+\sqrt{-1}\,q$ with $p,q\in{\mathbb R}$,
\begin{equation}\label{D(s)}
\begin{array}{ll}
D(s;k)=\left(a_0+ \left(\lambda_k^{(0)} \right)^{-1}p+\displaystyle\sum_{i=1}^n a_i\frac{p \left(p+a_i\lambda_k^{(i)} \right)+q^2}{ \left(p+a_i\lambda_k^{(i)} \right)^2+q^2}\right)\\ [5mm]
\qquad\qquad\qquad\phantom{blank}+
\sqrt{-1}\left( \left(\lambda_k^{(0)} \right)^{-1} q+\displaystyle\sum_{i=1}^n a_i\frac{a_i\lambda_k^{(i)}q}{ \left(p+a_i\lambda_k^{(i)} \right)^2+q^2}\right).
\end{array}
\end{equation}
The real part $\text{Re}\,D(s;\lambda)$ of $D(s;\lambda)$ is positive for $p\ge0,\,q\in{\mathbb R}$. Further, if $q\not=0$, the imaginary part of $D(s,\lambda)$ never vanishes. Hence, we have the conclusion.

\qed

\bigskip
The expression for $D(s;k)^{-1}$ has a unique expansion into partial fractions:
\begin{lemma}\label{lemA.4}
For each $1\le k\le \tilde d$, $D(s;k)\Pi_{i=1}^n (s+a_i \lambda_k^{(i)})=(\lambda_k^{(0)})^{-1}(s-r_1^k)^{j_1}\cdots(s-r_{m_k}^k)^{j_{m_k}}$ only has real roots $r_1^k,\cdots,r_{m_k}^k$, which are the roots of $D(s;k)=0$ with multiplicities $j_i$ such that $\sum_{l=1}^{m_k} j_l=n+1$. Moreover, we have the unique representation $D(s;k)^{-1}=\sum_{l=1}^{m_k}\sum_{q=1}^{j_l} g_{k,l,q} (s-r_l^k)^{-q}$ with some non-zero real constants $g_{k,l,q},\,1\le k\le \tilde d,\,1\le l\le n,\,1\le k\le j_l$.
\end{lemma}
\proof 
Define $Q_{n+1}(s;k)=\Pi_{l=1}^n (s+A_l^k) D(s;k)$ with $A_l^k:=a_l\lambda_k^{(l)}$. It is clear that $Q_{n+1}(s;k)$ is a polynomial of $s$ with degree $n+1$.
Also, $s=-A_l^k$ for $l=1,\cdots,n$ are not roots of $Q_{n+1}(s;k)$.
Thus, we conclude that $D(s;k)\Pi_{l=1}^n (s+A_l^k)=(\lambda_k^{(0)})^{-1}(s-r_1^k)^{j_1}\cdots(s-r_{m_k}^k)^{j_{m_k}}$, where $r_1^k,\cdots,r_{m_k}^k$ are roots of $D(s;k)=0$ and $\sum_{l=1}^{m_k} j_l=n+1$.

Since $D(s;k)^{-1}=(s-r_1^k)^{-j_1}\cdots(s-r_{m_k}^k)^{-j_{m_k}} \lambda_k^{(0)} \Pi_{l=1}^n  (s+A_l^k)$, there exist unique real constants $g_{k,l,q},\,1\le k\le \tilde d,\,1\le l\le n,\,1\le k\le j_l$
such that $D(s;k)^{-1}=\sum_{l=1}^{m_k}\sum_{q_l=1}^{j_l} g_{k,l,q_l}$ $(s-r_l^k)^{-q_l}$.

\qed 

\bigskip
The inverse Laplace transform $\mathcal{L}^{-1}[(\,\cdot-\alpha)^{-M}](t)$ of $(s-\alpha)^{-M}$ with $\alpha\in\R$, $M\in {\mathbb N}$ is $t^{M-1}e^{\alpha t}\left( (M-1)!\right)^{-1}1_+(t)$. 
Hence, by \eqref{A.11}, \eqref{A.12} and Lemma \ref{lemA.4}, we have
\begin{equation}\label{G(t)2}
\begin{array}{ll}
G(t)=\displaystyle\sum_{k=1}^{\tilde d} \mathcal{L}^{-1}(D(\cdot;k)^{-1})P_k\\
\qquad=\displaystyle\sum_{k=1}^{\tilde d}\left(\displaystyle\sum_{l^k=1}^{m_k}\displaystyle\sum_{q_{l^k}=1}^{j_{l^k}}g_{k,l^k,q_{l^k}}t^{q_{l^k}-1}e^{r_{l^k}^k t}((q_{l^k}-1)!)^{-1}1_+(t)\right) P_k. 
\end{array}
\end{equation}
Since each $r_{l^k}^k<0$ from Lemma \ref{lemA.3} and Lemma \ref{lemA.4}, we conclude that $G(t)$ decays exponentially as $t$ tends to $\infty$. Therefore, we obtain the following theorem:

\bigskip
\begin{Th}\label{homog relaxation tensor}
The relaxation tensor $G(t)$ for the homogeneous EBM is a tempered distribution in ${\mathbb R}$ given as \eqref{G(t)2}, which satisfies causality and is differentiable infinitely many times on $[0,\infty)$. Also, for $t\ge0$, $G(t)$ satisfies the full symmetry, strong convexity, and exponential decay properties.
\end{Th}

\bigskip
\begin{remark}\label{some hetero and without Maxwell}${}$
 Theorem 3.4 can be generalized to some extent for the case that the elements of each $C_l$ belong to the broken Sobolev space $\hat H^r(\Omega)$ with $r>d/2$ and $C_l$'s share the same interfaces compactly embedded in $\Omega$ (see \cite{LNQ}) by using the joint spectral measure (see Theorem 1 on page 154 of \cite{BS}).
\end{remark}

\section{Alternative construction of the relaxation tensor: Preliminaries}\label{idea}
\setcounter{equation}{0}

In this section, we first derive an ordinary differential system of equations for $(\psi,\phi)$ with $\phi=(\phi_1,\cdots,\phi_n)$, and then give an outline of the new method. After that, as a preliminary to the next section, we introduce some relevant tensor block matrices.

Suppose $e[\dot u]$ is given. Then from \eqref{ebm-prob} and \eqref{sigma}, we obtain 
\begin{align*}
\dot\psi&=e[\dot{u}]-\sum_{i=0}^n\dot\phi_i
=e[\dot{u}]-\eta_0^{-1}\sigma -\sum_{i=1}^n\eta_i^{-1}(\sigma -C_i\phi_i)\\
&=e[\dot{u}]-\left(\sum_{i=0}^n\eta_i^{-1}\right)\sigma +\sum_{i=1}^n\eta_i^{-1}C_i\phi_i
=e[\dot{u}]-\left(\sum_{i=0}^n\eta_i^{-1}\right)C_0\psi+\sum_{i=1}^n\eta_i^{-1}C_i\phi_i.
\end{align*}
Combining this with \eqref{ebm-prob}, we have the following system for $(\psi,\phi)$:
\begin{equation}\label{2.6}
\left\{
\begin{aligned}
\partial_t\psi&=e[\dot{u}]-\left(\sum_{i=0}^n\eta_i^{-1}\right)C_0\psi+\sum_{i=1}^n\eta_i^{-1}C_i\phi_i\\
\partial_t\phi_i&=\eta_i^{-1}C_0\psi-\eta_i^{-1}C_i\phi_i, \quad 1\leq i\leq n.
\end{aligned}
\right.
\end{equation}
Here, we recall that $C_0\psi=\sigma$ and $\phi_i=e_i,\,\,i=1,2,\cdots,n$. Then, from \eqref{2.6}, it is easy to derive again \eqref{integro-diff} using \eqref{assumption for integro-diff}. Therefore, \eqref{integro-diff} and \eqref{2.6} are equivalent under \eqref{assumption for integro-diff}.

Now, we write \eqref{2.6} as
\begin{equation}\label{tensor block 2.6}
\frac{d}{dt}
\begin{pmatrix}
\psi \\
\phi
\end{pmatrix}
=L_b(x)
\begin{pmatrix}
\psi \\
\phi
\end{pmatrix}+
\begin{pmatrix}
 e[\dot u]\\
 0
\end{pmatrix},
\end{equation}
where
\begin{equation}\label{2.8}
L_b(x):=
\begin{pmatrix}
                       -\left(\sum_{i=0}^n\eta_i^{-1}\right)C_0     
                       & \eta_1^{-1}C_1 & \cdots     & \eta_{n-1}^{-1}C_{n-1}            &    \eta_n^{-1}C_n    \\[0.5em]
                       
                             \eta_1^{-1}C_0                              & -\eta_1^{-1}C_1 & \cdots     &    0               &            0         \\[0.5em]
            \cdots                                   & \cdots                   & \cdots     & \cdots             & \cdots               \\[0.5em]
                               \eta_{n-1}^{-1}C_0                               & 0              & \cdots     & -\eta_{n-1}^{-1}C_{n-1}     &            0         \\[0.5em]
                              \eta_n^{-1}C_0                              & 0              & \cdots     &    0               &    -\eta_n^{-1}C_n
\end{pmatrix},
\end{equation}
and $L_b(x)$ can be considered as a linear map from $(\R^{d\times d}_\text{sym})^{n+1}$ into itself for a fixed a.e. $x\in\Omega$.
$L_b=L_b(x)$ is an example of a tensor block matrix (cf. \cite{Nikabadze}). In general, a tensor block matrix $\overline{\overline{\mathcal{H}}}$ which represents a linear map from $(\R^{d\times d}_\text{sym})^{n+1}$ into itself for a fixed a.e. $x\in\Omega$ is given as
\begin{equation}\label{tensor1}
\begin{aligned}
\overline{\overline{\mathcal{H}}}=
\begin{pmatrix}
\mathcal{H}_{11}& \mathcal{H}_{12} & \cdots & \mathcal{H}_{1(n+1)} \\
\mathcal{H}_{21}& \mathcal{H}_{22} & \cdots &  \mathcal{H}_{2(n+1)}\\
         \cdots &           \cdots & \cdots &                       \\
\mathcal{H}_{(n+1)1}& \mathcal{H}_{(n+1)2} & \cdots &  \mathcal{H}_{(n+1)(n+1)}
\end{pmatrix},
\end{aligned}
\end{equation}
where each
$
\mathcal{H}_{ij}=
\begin{pmatrix}
\mathrm{H}^{(ij)}_{pqrs}
\end{pmatrix}
$ is a full symmetric fourth-order tensor. Here, note that we are using the convention suppressing $\text{a.e.}\, x\in\Omega$ dependency for $\overline{\overline{\mathcal{H}}}$, which is applied not only for square tensor block matrices but also for non-square tensor block matrices. The transpose $\overline{\overline{\mathcal{H}}}^{\,\mathfrak{t}}$ of $\overline{\overline{\mathcal{H}}}$ is defined as
\begin{equation}\label{transpose}
\begin{aligned}
\overline{\overline{\mathcal{H}}}^{\mathfrak{t}}:=
\begin{pmatrix}
\mathcal{H}_{11}^{\mathfrak{t}}& \mathcal{H}_{21}^{\mathfrak{t}} & \cdots & \mathcal{H}_{(n+1)1}^{\mathfrak{t}} \\
\mathcal{H}_{12}^{\mathfrak{t}}& \mathcal{H}_{22}^{\mathfrak{t}} & \cdots &  \mathcal{H}_{(n+1)2}^{\mathfrak{t}}\\
         \cdots &           \cdots & \cdots &                       \\
\mathcal{H}_{1(n+1)}^{\mathfrak{t}}& \mathcal{H}_{2(n+1)}^{\mathfrak{t}} & \cdots &  \mathcal{H}_{(n+1)(n+1)}^{\mathfrak{t}}
\end{pmatrix}.
\end{aligned}
\end{equation}
Since each $\mathcal{H}_{ij}^{\,\mathfrak{t}}=\mathcal{H}_{ij}$,
we have $\overline{\overline{\mathcal{H}}}^{\,\mathfrak{t}}=\overline{\overline{\mathcal{H}}}$ if $\mathcal{H}_{ij}=\mathcal{H}_{ji}$ for each $1\le i,j\le n+1$.
We denote the set of such $\overline{\overline{\mathcal{H}}}$ by ${\mathbb M}$. For another tensor block matrix $\overline{\overline{\mathcal{J}}}$ similar to $\overline{\overline{\mathcal{H}}}$ given as
\begin{equation}
\begin{aligned}
\overline{\overline{\mathcal{J}}}=
\begin{pmatrix}
\mathcal{J}_{11}& \mathcal{J}_{12} & \cdots & & \mathcal{J}_{1(n+1)} \\
\mathcal{J}_{21}& \mathcal{J}_{22} & \cdots & & \mathcal{J}_{2(n+1)}\\
         \cdots &           \cdots & \cdots &                       \\
\mathcal{J}_{(n+1)1}& \mathcal{J}_{(n+1)2} & \cdots & & \mathcal{J}_{(n+1)(n+1)}
\end{pmatrix}
\end{aligned}
\end{equation}
with $\mathcal{J}_{ij}=
\begin{pmatrix}
\mathrm{J}^{(ij)}_{pqrs}
\end{pmatrix}$, we define the product $\overline{\overline{\mathcal{H}}}\,\overline{\overline{\mathcal{J}}}$ by giving its $il$-blockwise component $(\overline{\overline{\mathcal{H}}}\,\overline{\overline{\mathcal{J}}})_{il}$ as
\begin{equation}\label{tensor2}
\begin{pmatrix}
\overline{\overline{\mathcal{H}}}\,\overline{\overline{\mathcal{J}}}
\end{pmatrix}_{il}=
\sum_k\mathcal{H}_{ik}\,\mathcal{J}_{kl},\quad
\mathcal{H}_{ik}\mathcal{J}_{kl}=\sum_{\varepsilon\zeta}
\mathrm{H}^{(ik)}_{pq\varepsilon\zeta }\,\mathrm{J}^{(kl)}_{\varepsilon\zeta rs }
.
\end{equation}
Also, for
\begin{equation}\label{tensor3}
\underline{\mathcal{E}}=
\begin{pmatrix}
\mathcal{E}_{1}\\
\vdots\\
\mathcal{E}_{n+1}
\end{pmatrix}\,\,\text{with each matrix}\,\,
\mathcal{E}_{i}=
\begin{pmatrix}
e^{(i)}_{kl}
\end{pmatrix}\in{\mathbb R}_{\text{sym}}^{d\times d},
\end{equation}
we define the product $\overline{\overline{\mathcal{H}}}\,\underline{\mathcal{E}}$ by giving its $i$-blockwise component $\left(\overline{\overline{\mathcal{H}}}\,\underline{\mathcal{E}}\right)_i$ as
\begin{equation}\label{tensor4}
\begin{pmatrix}
\overline{\overline{\mathcal{H}}}\,\underline{\mathcal{E}}
\end{pmatrix}_{i}=
\sum_k\mathcal{H}_{ik}\,\mathcal{E}_{k},\quad
\mathcal{H}_{ik}\mathcal{E}_{k}=\sum_{\varepsilon,\zeta}
\mathrm{H}^{(ik)}_{pq\varepsilon\zeta }\,e^{(k)}_{\varepsilon\zeta }.
\end{equation}
We denote by ${\mathbb V}$ the set of all $\underline{\mathcal{E}}$ given as \eqref{tensor3}.
Finally, we define the inner product $\left\langle\underline{\mathcal{E}},\underline{\mathcal{F}}\right\rangle$ as
\begin{equation}
\left\langle\underline{\mathcal{E}},\underline{\mathcal{F}}\right\rangle:=\sum_k \mathcal{E}_k:\mathcal{F}_k,
\end{equation}
where $\underline{\mathcal{F}}$ is similar to $\underline{\mathcal{E}}$ given as
$$\underline{\mathcal{F}}=
\begin{pmatrix}
\mathcal{F}_{1}\\
\vdots\\
\mathcal{F}_{n+1}
\end{pmatrix}\,\,\text{with each matrix}\,\,
\mathcal{F}_{i}=
\begin{pmatrix}
f^{(i)}_{kl}
\end{pmatrix}\in {\mathbb R}_{\text{sym}}^{d\times d},
$$
and  $\mathcal{E}_k:\mathcal{F}_k=\sum_{\varepsilon,\zeta}e^{(k)}_{\varepsilon\zeta}f^{(k)}_{\varepsilon\zeta}$ is the Frobenius inner product defined before.
Then, we say $\overline{\overline{\mathcal{H}}}$ is positive if there exists $\gamma>0$ such that $\left\langle \overline{\overline{\mathcal{H}}}\,\underline{\mathcal{E}},\,\underline{\mathcal{E}}\right\rangle\ge \gamma\Vert\underline{\mathcal{E}}\Vert^2$ a.e. $x\in\Omega$ for all $\underline{\mathcal{E}}\not=0$, and write this as $\overline{\overline{\mathcal{H}}}\ge\gamma\,\overline{\overline{\mathcal{I}}}$, with the tensor block identity matrix $\overline{\overline{\mathcal{I}}}\in{\mathbb M}$ and $\Vert\underline{\mathcal{E}}\Vert:=\sqrt{\left\langle\underline{\mathcal{E}},\underline{\mathcal{E}}\right\rangle}$. Similarly, we define $\gamma'\,\overline{\overline{\mathcal{I}}}\ge\overline{\overline{\mathcal{H}}}$ with $\gamma'>0$. Also, we denote by ${\mathbb M}_+$ the set of all $\overline{\overline{\mathcal{H}}}\in{\mathbb M}$ such that $\overline{\overline{\mathcal{H}}}\ge\gamma\,\overline{\overline{\mathcal{I}}}$ with $\gamma>0$ depending on $\overline{\overline{\mathcal{H}}}$.

The brief outline for obtaining a causal relaxation tensor $G(t)$ is as follows. 
We first give factorizations of some matrices into matrices in ${\mathbb M}$. To begin with, let 
\begin{equation}\label{C and D}
\overline{\overline{\mathcal{C}}}:=\text{diag}(C_0,C_1,\cdots,C_n),\,\,\,\overline{\overline{\mathcal{D}}}:=\sqrt{\overline{\overline{\mathcal{C}}}}\in{\mathbb M}_+,
\end{equation}
where $\text{diag}(C_0,C_1,\cdots,C_n)$ is the tensor block diagonal matrix with block diagonal components $C_0, C_1,\cdots,C_n$. Moreover we define $\overline{\overline{\mathcal{L}}}\in{\mathbb{M}}$ by
\begin{equation}\label{L}
\overline{\overline{\mathcal{L}}}:=\text{diag}\left(-\displaystyle\sum_{i=0}^n\eta_i^{-1}I_d, -\eta_1^{-1}I_d,\cdots,-\eta_n^{-1}I_d\right)+\overline{\overline{\mathcal{M}}}+\overline{\overline{\mathcal{M}}}^{\mathfrak{t}},
\end{equation}
where 
$$
\overline{\overline{\mathcal{M}}}:=
\begin{pmatrix}
O_d,\eta_1^{-1} I_d,\cdots,\eta_n^{-1}I_d\\
\\
\quad \text{\it\Large O}\quad\\
\\
\end{pmatrix}
$$
with fourth-order zero tensor $O_d\in \R^{d\times d\times d\times d}$, fourth-order identity tensor $I_d\in \R^{d\times d\times d\times d}$, and $\left({\mathbb R}^{d\times d\times d\times d}\right)^{n-1}\times \left({\mathbb R}^{d\times d\times d\times d}\right)^{n}$ zero tensor $\text{\it\large O}$. Then, $L_b$ and $\overline{\overline{\mathcal{C}}}\,e^{tL_b}$ have the the following factorizations:
\begin{equation}\label{factorization}
\left\{
\begin{array}{ll}
L_b=\overline{\overline{\mathcal{L}}}\,\overline{\overline{\mathcal{C}}},\\
\\
\overline{\overline{\mathcal{C}}}e^{tL_b}=\overline{\overline{\mathcal{D}}}\,e^{t\overline{\overline{\mathcal{A}}}}\,\overline{\overline{\mathcal{D}}}\,\,\,\text{with}\,\,\overline{\overline{\mathcal{A}}}:=\overline{\overline{\mathcal{D}}}\,\overline{\overline{\mathcal{L}}}\,\overline{\overline{\mathcal{D}}}.
\end{array}
\right.
\end{equation}
For any fixed  a.e. $x\in\Omega$, we first obtain causal $\overline{\overline{\mathcal{E}}}\in\mathcal{D}'({\mathbb R};{\mathbb M})$ which satisfies
\begin{equation}\label{C1}
\left(\frac{d}{dt}-\overline{\overline{\mathcal{A}}}\right)\,\overline{\overline{\mathcal{E}}}=\delta(t)\,\overline{\overline{\mathcal{I}}}.
\end{equation}
Then, any solution $\underline{\mathcal{U}}\in\mathcal{D}'({\mathbb R};{\mathbb V})$ of 
\begin{equation}\label{eq_U}
\left(\frac{d}{dt}-\overline{\overline{\mathcal{A}}}\right)\,\underline{\mathcal{U}}=\underline{\mathcal{F}},\,\,\,\underline{\mathcal{F}}:=
\begin{pmatrix}
e[\dot{u}]\\
0
\end{pmatrix}
\end{equation}
is given as
\begin{equation}\label{E convolution F}
\underline{\mathcal{U}}(t)=\left(\overline{\overline{\mathcal{E}}}\ast\underline{\mathcal{F}}\right)(t),
\end{equation}
where we have assumed that $\underline{\mathcal{F}}$ is right-sided.
From \eqref{sigma} and \eqref{E convolution F}, we extract $G(t)$
as
\begin{equation}\label{extraction of G(t)}
G(t)=\text{the $(1,1)$ tensor block component of $\overline{\overline{\mathcal{D}}}\,\overline{\overline{\mathcal{E}}}(t)\,\overline{\overline{\mathcal{D}}}$}. 
\end{equation}

\section{Relaxation tensor for the EBM, its properties, and the BVS}\label{new method}\setcounter{equation}{0}

In this section, we first give the details of the proof for obtaining a causal relaxation tensor $G(t)$ outlined in the previous section. After that, we show that the relaxation tensor has properties which satisfy Assumption \ref{exponential} for the exponential decay property of solutions to the initial boundary value problem for BVS given as \eqref{ebm-prob2}.

As for the derivation of $G(t)$, since \eqref{factorization} and \eqref{extraction of G(t)} are easy to prove, we will focus on obtaining $\overline{\overline{\mathcal{E}}}$. To begin with, for a.e. $x\in\Omega$, consider the Cauchy problem
\begin{equation}\label{Cauchy prob}
\left(\frac{d}{dt}-\overline{\overline{\mathcal{A}}}\right)\,\overline{\overline{\mathcal{E}}}_f=0\,\,\,\text{in ${\mathbb R}$},\,\,\,\overline{\overline{\mathcal{E}}}_f=\overline{\overline{\mathcal{I}}}\,\,\,\text{at $t=0$}.
\end{equation}
Of course, $\overline{\overline{\mathcal{E}}}_f\in C^{\infty}({\mathbb R}; {\mathbb M})$ is given as 
\begin{equation}\label{tilde E}
\overline{\overline{\mathcal{E}}}_f(t):=e^{t\,\overline{\overline{\mathcal{A}}}},\,\,\,t\in{\mathbb R},\,\,\text{a.e.\,$x\in\Omega$}.
\end{equation}
Define $\overline{\overline{\mathcal{E}}}\in\mathcal{D}'({\mathbb R};{\mathbb M})$, a.e. $x\in\Omega$ by
\begin{equation}\label{def E}
\overline{\overline{\mathcal{E}}}(t)=\overline{\overline{\mathcal{E}}}_f(t)\,\,(t\ge0),\,\,0\,\,(t<0).
\end{equation}

\begin{lemma}\label{eq_delta}
$\overline{\overline{\mathcal{E}}}\in\mathcal{D}'({\mathbb R};{\mathbb M})$ satisfies
\begin{equation}\label{pairing}
\left[\left(\frac{d}{dt}-\overline{\overline{\mathcal{A}}}\right)\,\overline{\overline{\mathcal{E}}},\Phi\right]=\left[\delta(t)\,\overline{\overline{\mathcal{I}}},\Phi\right]
\end{equation}
for any $\Phi\in\mathcal{D}({\mathbb R};{\mathbb{M}})$, where $[\,\,\,,\,\,]$ denotes the dual pairing of $\mathcal{D}'({\mathbb R};{\mathbb M})$ and $\mathcal{D}({\mathbb R};{\mathbb M})$ by identifying ${\mathbb M}$ with its dual space via the standard inner product $\left\langle\,\,\,\,,\,\,\right\rangle_{{\mathbb M}}$.
\end{lemma}
\proof Observe that 
\begin{align*}
\left[\frac{d}{dt}\,\overline{\overline{\mathcal{E}}},\Phi\right]
&=-\int_{\mathbb R}\left\langle\overline{\overline{\mathcal{E}}}(t),\frac{d}{dt}\Phi(t)\right\rangle_{{\mathbb M}}\,dt=-\int_0^\infty\left\langle\overline{\overline{\mathcal{E}}}_f(t),\frac{d}{dt}\Phi(t)\right\rangle_{{\mathbb M}}\,dt
\\[1em]
&=\left.\left(\left\langle\overline{\overline{\mathcal{E}}}_f,\Phi(t)\right\rangle_{\mathbb M}\right)\right\vert_{t=0}+
\int_0^\infty\left\langle\frac{d}{dt}\,\overline{\overline{\mathcal{E}}}_f(t),\Phi(t)\right\rangle_{{\mathbb M}}\,dt
\\[1em]
&=\left\langle\overline{\overline{\mathcal{I}}},\Phi(0)\right\rangle_{\mathbb M}+
\int_0^\infty\left\langle\frac{d}{dt}\,\overline{\overline{\mathcal{E}}}_f(t),\Phi(t)\right\rangle_{{\mathbb M}}\,dt.
\end{align*}
Hence, we have
$$
\left[\left(\frac{d}{dt}-\overline{\overline{\mathcal{A}}}\right)\,\overline{\overline{\mathcal{E}}},\Phi\right]=\left\langle\overline{\overline{\mathcal{I}}},\Phi(0)\right\rangle_{\mathbb M}+
\int_0^\infty\left\langle\left(\frac{d}{dt}-\overline{\overline{\mathcal{A}}}\right)\,\overline{\overline{\mathcal{E}}}_f(t),\Phi(t)\right\rangle_{{\mathbb M}}\,dt=\left[\delta(t)\overline{\overline{\mathcal{I}}},\Phi\right],
$$
which implies \eqref{pairing}.
\qed

\medskip
Next, we will derive the properties of the relaxation tensor. We first state the following proposition:

\begin{Prop} There exist constants $\alpha_1>\alpha_2 >0$ independent of a.e.$\,x\in\Omega$ such that
\begin{equation}\label{bound of A}
\alpha_2\,\overline{\overline{\mathcal{I}}}\le-\overline{\overline{\mathcal{A}}}\le\alpha_1\overline{\overline{\mathcal{I}}}\,\,\,\text{a.e. $x\in\Omega$}.    
\end{equation}
\end{Prop}

\bigskip
Since 
\begin{equation}\label{bound of D}
    \beta_1\overline{\overline{\mathcal{I}}}\le\overline{\overline{\mathcal{D}}}\le\beta_2\overline{\overline{\mathcal{I}}}\,\,\,\text{a.e. $x\in\Omega$}
\end{equation}
for some constants $0<\beta_1<\beta_2$ independent of a.e.$\,x\in\Omega$, this easily follows from the identity
\begin{equation}    
\left\langle\left(\overline{\overline{\mathcal{C}}}L_b\right)\underline{\mathcal{Y}},\,\underline{\mathcal{Y}}\right\rangle 
=\left\langle\left(\overline{\overline{\mathcal{D}}}\,\overline{\overline{\mathcal{A}}}\,\overline{\overline{\mathcal{D}}}\right)\,\underline{\mathcal{Y}},\,\underline{\mathcal{Y}}\right\rangle\ 
=\left\langle\left(\overline{\overline{\mathcal{A}}}\,\overline{\overline{\mathcal{D}}}\right)\,\underline{\mathcal{Y}},\,\overline{\overline{\mathcal{D}}}\,\underline{\mathcal{Y}}\right\rangle
,\,\,\,\underline{\mathcal{Y}}\in{\mathbb{V}},\,\,\text{a.e. $x\in\Omega$}
\end{equation}
and the following lemma:
    
\begin{lemma}\label{identity for Lb}
For any $\underline{\mathcal{Y}}=(\psi,\phi)^{\mathfrak{t}}\in{\mathbb V}$, we have
\begin{equation}\label{C}
\left\langle -\big(\overline{\overline{\mathcal{C}}}L_b\big)\underline{\mathcal{Y}},\,\underline{\mathcal{Y}}\right\rangle=\eta_0^{-1}| C_0\psi|^2+\displaystyle\sum_{i=1}^n \eta_i^{-1} |C_0\psi-C_i\phi_i|^2\,\,\,\text{a.e. $x\in\Omega$},
\end{equation}
which immediately yields
\begin{equation}\label{computation for identity}
\left\langle -\big(\overline{\overline{\mathcal{C}}}L_b\big) \underline{\mathcal{Y}},\, \underline{\mathcal{Y}}\right\rangle\ge c \Vert\underline{\mathcal{Y}}\Vert^2\,\,\,\text{a.e.\,$x\in\Omega$}
\end{equation}
for some constant $c>0$ independent of $\underline{\mathcal{Y}}\in{\mathbb V}$ and a.e. $x\in\Omega$.
\end{lemma}
\proof It is enough to prove \eqref{C}. By the definition of $L_b$, we have
\begin{eqnarray*}
-\left\langle (\overline{\overline{\mathcal{C}}}L_b)\underline{\mathcal{Y}},\,\underline{\mathcal{Y}}\right\rangle \!\!
&=& \!\! \left\{ \left(\sum_{i=0}^n\eta_i^{-1} \right)C_0\psi -\sum_{i=1}^n\eta_i^{-1}C_i\phi_i \right\} : (C_0\psi)
-\sum_{i=1}^n \left\{ \eta_i^{-1}(C_0\psi-C_i\phi_i)\right\} : (C_i\phi_i)\\
&=&\left(\sum_{i=0}^n\eta_i^{-1}\right)|C_0\psi|^2 
-\sum_{i=1}^n\eta_i^{-1}\{(C_i\phi_i):(C_0\psi)\}
\\
& &\qquad\qquad\qquad\qquad\qquad\,\,\,-\sum_{i=1}^n\eta_i^{-1}\{( C_0\psi):(C_i\phi_i)\}
+\sum_{i=1}^n\eta_i^{-1}|C_i\phi_i|^2\\
&=&\eta_0^{-1}|C_0\psi|^2
+\sum_{i=1}^n\eta_i^{-1}|C_0\psi -C_i\phi_i|^2.
\end{eqnarray*}

Now, setting $\alpha:=\frac{1}{2}\eta_0^{-1}(\sum_{i=1}^n\eta_i^{-1})^{-1}$ and
$\beta:=\sqrt{1+\alpha}$ and using the inequality
\[
2(C_0\psi): (C_i\phi_i)=2(\beta C_0\psi):( \beta^{-1}C_i\phi_i)
\le \beta^2|C_0\psi|^2 +\beta^{-2}|C_i\phi_i|^2,
\]
we obtain 
\begin{align*}
\eta_0^{-1}|C_0\psi|^2
&+\sum_{i=1}^n\eta_i^{-1}|C_0\psi -C_i\phi_i|^2\\
=~&\frac{1}{2}\eta_0^{-1}|C_0\psi|^2
+\sum_{i=1}^n\eta_i^{-1}\left(\alpha |C_0\psi|^2+|C_0\psi -C_i\phi_i|^2\right)\\
=~&\frac{1}{2}\eta_0^{-1}|C_0\psi|^2
+\sum_{i=1}^n\eta_i^{-1}\left((1+\alpha) |C_0\psi|^2
-2(C_0\psi, C_i\phi_i)+|C_i\phi_i|^2\right)\\
\ge ~&\frac{1}{2}\eta_0^{-1}|C_0\psi|^2
+\sum_{i=1}^n\eta_i^{-1}\left((1+\alpha) |C_0\psi|^2
-\beta^2|C_0\psi|^2 -\beta^{-2}|C_i\phi_i|^2+|C_i\phi_i|^2\right)\\
= ~&\frac{1}{2}\eta_0^{-1}|C_0\psi|^2
+\sum_{i=1}^n\eta_i^{-1}\left(\frac{\alpha}{1+\alpha}|C_i\phi_i|^2\right).
\end{align*}
\qed

\medskip
For estimating $\overline{\overline{\mathcal{E}}}_f$, we provide the following lemma:
\begin{lemma}\label{estimate_exp}
\begin{equation}\label{tilde E_estimate}
e^{-\alpha_1 t}\overline{\overline{\mathcal{I}}}\le\overline{\overline{\mathcal{E}}}_f(t)\le e^{-\alpha_2 t} \overline{\overline{\mathcal{I}}}\,\,\,\,\text{a.e. $x\in\Omega$}
\end{equation}
holds.
\end{lemma}
\begin{proof}
By \eqref{bound of A}, $\alpha_2\,\overline{\overline{\mathcal{I}}}\le-\overline{\overline{\mathcal{A}}}\le\alpha_1\,\overline{\overline{\mathcal{I}}}\,\,$ a.e.$\,x\in\Omega$. This yields
$$
e^{\alpha_2t}\,\overline{\overline{\mathcal{I}}}\le e^{-t\overline{\overline{\mathcal{A}}}}=\displaystyle\sum_{j=0}^\infty (j!)^{-1}t^j(-\overline{\overline{\mathcal{A}}})^j\le e^{\alpha_1t}\,\overline{\overline{\mathcal{I}}}\,\,\,\,\text{a.e.$\,x\in\Omega$}.
$$
Then, since 
$$0<m_1\overline{\overline{\mathcal{I}}}\le\overline{\overline{\mathcal{M}}}\le m_2\,\overline{\overline{\mathcal{I}}}\Longrightarrow
m_2^{-1}\,\overline{\overline{\mathcal{I}}}\le\overline{\overline{\mathcal{M}}}^{-1}\le m_1^{-1}\,\overline{\overline{\mathcal{I}}}$$
for some positive numbers $m_1<m_2$, due to $\Vert\underline{\mathcal{Y}}\Vert^2=\left[\overline{\overline{\mathcal{M}}}^{\,2}(\overline{\overline{\mathcal{M}}}^{-1}\underline{\mathcal{Y}}),\overline{\overline{\mathcal{M}}}^{-1}\underline{\mathcal{Y}}\right]$ for any $\underline{\mathcal{Y}}\in{\mathbb V}$, we immediately have \eqref{tilde E_estimate}.
\end{proof}

\bigskip

\noindent
Via \eqref{extraction of G(t)} and the full symmetry of $\overline{\overline{\mathcal{E}}}(t)$ (see \eqref{factorization}, \eqref{tilde E}, \eqref{def E}), we can list properties of $G(t)$ which follow from those of $\overline{\overline{\mathcal{E}}}(t)$:

\medskip\medskip

\noindent
\textbf{Properties of the relaxation tensor}. \textit{
\begin{itemize}
\item [{\rm (i)}] Causality of $G(t)$ follows from that of $\overline{\overline{\mathcal{E}}}(t)$.
\item [{\rm (ii)}] Positivity and full symmetry of $G(t)$ follow from those of $\overline{\overline{\mathcal{E}}}(t)$ and $\overline{\overline{\mathcal{D}}}$.
\item [{\rm (iii)}] $G(t)\in C^\infty([0,\infty);L^\infty(\Omega;{\mathbb R}^{d\times d\times d\times d}))$ follows from that of $\overline{\overline{\mathcal{E}}}(t)\big|_{t\ge0}=\overline{\overline{\mathcal{E}}}_f(t)$.
\item [{\rm (iv)}] The estimates
\begin{equation}\label{estimates of G(t)}
\left\{
\begin{array}{ll}
\beta_1^2\alpha_2^{2j}e^{-\alpha_1 t}I_d\le G^{(2j)}(t)\le\beta_2^2\alpha_1^{2j}e^{-\alpha_2 t}I_d\,\,\text{a.e.}\,x\in\Omega,\\
\\
-\beta_2^2\alpha_1^{2j+1}e^{-\alpha_1 t}I_d\le G^{(2j+1)}(t)\le-\beta_1^2\alpha_2^{2j+1}e^{-\alpha_2 t}I_d\,\,\text{a.e.}\,x\in\Omega
\end{array}
\right.    
\end{equation}
hold for $j=0,1,\cdots$, where $G^{(k)}(t)$ denote the $k$-th derivative of $G(t)$ with respect to time. These estimates follow from \eqref{bound of A}, \eqref{bound of D}, and the following fact:
$$
\left\{
\begin{array}{ll}
\left(\frac{d}{dt}\right)^{2j}e^{t\overline{\overline{\mathcal{A}}}}=(-\overline{\overline{\mathcal{A}}})^je^{-t(-\overline{\overline{\mathcal{A}}})}(-\overline{\overline{\mathcal{A}}})^j,\\
\\
\left(\frac{d}{dt}\right)^{2j+1}e^{t\overline{\overline{\mathcal{A}}}}=-(-\overline{\overline{\mathcal{A}}})^{j+1/2}e^{-t(-\overline{\overline{\mathcal{A}}})}(-\overline{\overline{\mathcal{A}}})^{j+1/2}
\end{array}
\right.
$$
for $j=0,1,\cdots$.
\end{itemize}
}

\begin{remark}
Positivity of the relaxation tensor obtained by the usual method does not directly follow from \eqref{G(t)2}. Nevertheless, we have seen that \eqref{integro-diff} together with \eqref{star1} and \eqref{star2} is equivalent to \eqref{2.6}, so the usual method and that of \eqref{extraction of G(t)} produce the same relaxation tensors.
Hence, the relaxation tensor obtained by the usual method also satisfies positivity.
    
\end{remark}

\bigskip
From these properties of the relaxation tensor, it is almost immediate to see that the relaxation tensor $G(t)$ satisfies Assumption \ref{exponential}. Hence, we have the following theorem:
\begin{Th}
The BVS with the derived $G(t)$ has the exponential decay property for its solutions of the initial boundary value problem \eqref{ebm-prob2}.    
\end{Th}

\section{Discussion}

We introduced a new method for constructing the relaxation tensor for the EBM in the most general setting. That is, we only assume that the elastic tensors of the springs and viscosity of the dashpots are bounded measurable in the reference domain $\Omega$, and the tensors are anisotropic satisfying the full symmetry and strong convexity conditions. Furthermore, our method enabled to prove the sufficient properties needed to guarantee the exponential decay property of solutions of the initial boundary value problem for the VBS with mixed-type homogeneous boundary conditions and without external force. 

The method presented, here, can be applied to other spring-dashpot models, such as the extended Maxwell model, the standard linear solid model, and the Kelvin-Voigt model in the most general setting. The EBM is the most cumbersome model for which to construct the relaxation tensor. 
For the L-method to apply, we need commutativity of the elasticity tensors.


The Maxwell component in the EBM appears to be essential. If we were to omit it, the following would happen. From \eqref{sigma} and \eqref{2.6}, we would get
$$
\begin{aligned}
e[u]=\sum_{j=1}^n\phi_j,\,\,\,\,\,
\partial_t\phi_i=\eta_i^{-1}\sigma-\eta_i^{-1}C_i\phi_i,\,\,\,1\leq i\leq n.
\end{aligned}
$$
By expressing each $\phi_i$ in terms of $\sigma$, we then straightforwardedly find that
$$
\begin{aligned}
\hat{\sigma}(s)=s^{-1}\bigg(\sum_{i=1}^n(sI+\eta_i^{-1}C_i)^{-1}\bigg)^{-1}(s\hat{e}(s)).
\end{aligned}
$$
Thus, due to the presence of the factor $s^{-1}$ on the right-hand side of this equation, the relaxation tensor could not decay exponentially.

We claim that we obtained a complete description of the most general EBM with implied properties which are natural in the physics context.

\medskip

\bigskip
{\bf Acknowledgements:} 
The first author was supported by the Simons Foundation under the MATH + X program, the National Science Foundation under grant DMS-2108175, and the corporate members of the Geo-Mathematical Imaging Group at Rice University. The third author was partially supported by the Ministry of Science and Technology of Taiwan.
The fourth author was partially supported by JSPS KAKENHI (Grant Nos. JP22K03366). The fifth author was partially supported by JP19K03559.

\bibliographystyle{initials}
\bibliography{references}

\end{document}